\documentclass{amsart}
\usepackage{amssymb} 
\usepackage{amsmath} 
\usepackage{mathrsfs} 
\usepackage{eufrak}
\usepackage{amscd}
\usepackage{amsbsy}
\usepackage{comment}
\usepackage{hyperref}
\usepackage{mathtools}
\usepackage[matrix,arrow]{xy}

\usepackage{esint}
\usepackage[letterpaper, margin=1.1in]{geometry}
\usepackage{fancyhdr}
\usepackage[shortlabels]{enumitem}
\usepackage{xcolor}
\def\Xint#1{\mathchoice
{\XXint\displaystyle\textstyle{#1}}%
{\XXint\textstyle\scriptstyle{#1}}%
{\XXint\scriptstyle\scriptscriptstyle{#1}}%
{\XXint\scriptscriptstyle\scriptscriptstyle{#1}}%
\!\int}
\def\XXint#1#2#3{{\setbox0=\hbox{$#1{#2#3}{\int}$ }
\vcenter{\hbox{$#2#3$ }}\kern-.6\wd0}}

\def\dashint{\Xint-}

\newtheorem{thm}{Theorem}[section]

\newtheorem{remark}[thm]{Remark}

\newtheorem*{thm*}{Theorem}

\theoremstyle{definition}

\newcommand{\R}{{\mathbb R}}

\newcommand{\inv}{^{-1}}

\newcommand{\ep}{\epsilon}

\newcommand{\Rm}{\mathrm{Rm}}
\newcommand{\Rc}{\mathrm{Rc}}
\newcommand{\Vol}{\mathrm{Vol}}
\newcommand{\inj}{\mathrm{inj}}

\newcommand{\dV}{\mathrm{d}V}

\newcommand{\RNum}[1]{\uppercase\expandafter{\romannumeral #1\relax}}

\let\epsilon\varepsilon

\begin{document}

\setlength\parindent{15pt}

\vspace{-3em}

        \author{Albert Chau$^1$}
        \address{Department of Mathematics,
                The University of British Columbia, Room 121, 1984 Mathematics
                Road, Vancouver, B.C., Canada V6T 1Z2} \email{chau@math.ubc.ca}

        \thanks{$^1$Research
                partially supported by NSERC grant no. \#327637-06}

        \author[A. Martens]{Adam Martens}
        \address{Department of Mathematics, The University of British Columbia,
1984 Mathematics Road, Vancouver, B.C.,  Canada V6T 1Z2}
\email{martens@math.ubc.ca}


\begin{abstract}
In \cite{ChauMartens} the authors proved the long-time existence of Ricci flow starting from complete bounded curvature Riemannian manifolds with scale-invariant integral curvature bounded by a dimensional constant times the inverse of the Sobolev constant.   We generalize this result by replacing the bounded curvature assumption with the assumption that $g$ is only equivalent to a complete bounded curvature metric $h$ while satisfying a Morrey-type condition on the gradient of $g$ relative to $h$: a local integral condition on the covariant derivative $\nabla_h g$. The Morrey-type condition was first considered in \cite{LeeLiu} in the context of Ricci flow on non-compact manifolds, and in particular allows the possibility for $g$ to have unbounded curvature on $M$.  As in \cite{ChauMartens}, our long-time solution enjoys curvature decay estimates implying in particular that  $M$ is diffeomorphic to $\mathbb{R}^n$.

\end{abstract}

        \bibliographystyle{amsplain}
\title{A note on Ricci flow from small curvature concentration and a Morrey-type condition}

\maketitle

\pagestyle{plain}

\vspace{-3em}

\section{Introduction}

In this note, we investigate the Ricci flow
\begin{equation}\label{RF}
                \left\{
                \begin{aligned}
                        \frac{{\partial }}{{\partial t}}g(t) &= -2 \text{Rc}_{g(t)}, \\
                        g(0) &= g
                \end{aligned}
                \right.
\end{equation}
on an initial complete $n$-dimensional Riemannian manifold $(M^n,g)$ under the initial assumptions of a global $L^2$-Sobolev inequality 
\begin{equation}\label{Sob}
    \left( \int_M |u|^{\frac{2n}{n-2}}\, \dV_g \right)^{\frac{n-2}{n}}
\;\le\;
C_g \int_M |\nabla u|_g^2\,\dV_g
\end{equation}
for all compactly supported smooth functions $u\in C^\infty_0(M)$, and also that the \textit{curvature concentration}, i.e., scale-invariant integral curvature, is relatively small compared to the Sobolev constant $C_g$:
\begin{equation}\label{CC}
    \left( \int_M | \Rm |_g^{\frac{n}{2}}\,\dV_g \right)^{\frac{2}{n}}
\le
\sigma(n,C_g).
\end{equation}
Here $\sigma$ is a constant which is relatively small compared to $C_g$ in a way we will make precise.  Here and in what follows $|\Rm|_g$ denotes the scalar norm relative to $g$ of the Riemann curvature tensor $\Rm(g)$, $\mathrm{d}V_g$ denotes the volume form relative to $g$, and we also define $\|\Rm\|_g := \sup_M |\Rm|_g$.  In general, given a tensor $T$ on $M$ we will define the quantities $|T|_g$ and $\|T\|_g$ similarly.

The conditions \eqref{Sob} and \eqref{CC} above pair naturally from a Ricci flow perspective, as we now briefly describe.  Together, the conditions imply a bound on Perelman's $\nu$-functional at time $t=0$, and this bound is generally preserved along the Ricci flow (at least for complete bounded curvature solutions).  Now in many situations, the curvature concentration can be shown to grow slowly along the flow, and this allows one to convert the bound on $\nu(t)$ back to a comparable Sobolev inequality at time $t$.  In this sense, the conditions \eqref{Sob} and \eqref{CC} are roughly preserved along the flow.  This mechanism can, under certain conditions, lead to long-time existence results for Ricci flow with parabolic control of the solution. Unfortunately, the conditions \eqref{Sob} and \eqref{CC} alone are (so far, at least) not sufficient to guarantee the existence of the Ricci flow solution and (semi)preservation of the assumed conditions; for this one needs to assume  some \textit{other} condition on the initial manifold (which may well be superfluous). The following theorem gives a unified formulation of results of this type appearing in several works under various such other conditions. In the cited works the statements appear in slightly different forms according to the context and applications there.

\begin{thm*}[\cite{ChenEric, ChauMartens, ChanChenLee, ChanHuangLee,  MartensScalar,  LeeLee}]
For any $n,A>0$, there exists $C_0(n,A),\sigma(n,A)>0$ such that the following holds. Let $(M^n,g)$ be a complete n-dimensional Riemannian manifold satisfying all the following estimates:

\,

\begin{enumerate}[(i)]
\item  
Condition \eqref{Sob} holds with $C_g=A$.

\item  
Condition \eqref{CC} holds with $\sigma(n, C_g)=\sigma(n, A)$.
\item 
Any one of the conditions (iii)' \cite{ChenEric, ChauMartens, ChanChenLee, ChanHuangLee,  MartensScalar,  LeeLee} below
holds.
\end{enumerate}

\,

 Then there exists a complete long-time Ricci flow $g(t)$, $t\in [0,\infty)$ with initial condition $g(0)=g$ which satisfies all the following scaling-invariant estimates on the solution:
\begin{equation}\label{generalconclusions}
\begin{cases}
&\|\Rm\|_{g(t)}\leq C_0 \ep/ t, \\
&\inj_{g(t)}\geq C_0\inv \sqrt t,\\
&\bar\nu(M,g(t))\geq -C_0, \text{ and}\\
&\displaystyle \left(\int_{M} |\Rm|_{g(t)}^{\frac{n}{2}}\,\dV_{g(t)}\right)^{\frac{2}{n}}\leq C_0\ep.
\end{cases}
\end{equation}

\end{thm*}

The following is a summary of the variety of condition (iii)'s appearing in the Theorem above, listed in chronological order.  
 Some of the conditions below are clear generalizations of others.

   \[
    (M^n,g) \text{ is \textit{Asymptotically Euclidean}}\footnote{We direct the reader to \cite{ChenEric} directly for the definition of asymptotically Euclidean used in this context.}\tag{iii' \cite{ChenEric}}
    \]
    
    \[
    \Rc_g\geq 0, \; \text{ and } \; (M^n,g) \text{ has bounded curvature}  \tag{iii' \cite{ChanChenLee}}
    \]
    \[
    \inf_{x\in M, r>0}\frac{\Vol_{g}B_g(x,r)}{r^n}\leq \mathfrak{v}, \; \mathrm{Scal}_g\geq 0, \; \text{and} \; \Rc_g>-\infty \tag{iii' \cite{ChanHuangLee}}
    \]

      \[
    (M^n,g) \text{ has bounded curvature}\tag{iii' \cite{ChauMartens}}
    \]

    \[
    \inf_{x\in M, r>0}\frac{\Vol_{g}B_g(x,r)}{r^n}\leq \mathfrak{v} \; \text{ and } \; \Rc_g>-\infty \tag{iii' \cite{MartensScalar}}
    \]

    \[
    \inf_{x\in M, r>0}\frac{\Vol_{g}B_g(x,r)}{r^n}\leq \mathfrak{v} \tag{iii' \cite{LeeLee}}
    \]
  We remark that the the estimates of the flow in \eqref{generalconclusions} may be used to apply the main result of \cite{ChauMartens} to any positive time (and thus, bounded curvature) metric $g(t)$  to obtain faster than $c/t$ curvature decay of the long-time Ricci flow
\[
\lim_{t\to\infty} 
\left( t\, \| \Rm\|_{g(t)}\right)
=0.
\]
This implies $\frac{1/4}{t}$-curvature decay of the Ricci flow solution starting from some positive time metric $g(T)$. This, combined with the injectivity radius estimate in \eqref{generalconclusions} implies – by the main result of \cite{HuangPeng} – that $M^n$ is diffeomorphic to $\mathbb{R}^n$.\\
 
The above results have been approached from two noticeably different perspectives. The perspective taken in \cite{ChanChenLee, ChanHuangLee,  MartensScalar,  LeeLee} can be described as being \begin{it}Local to Global\end{it}.  The approach begins by establishing \textit{localized} results, in which the hypotheses of the theorem ((i) and (ii) in particular) are assumed to hold only locally (that is, up to a fixed scale). In this case, one obtains short-time existence of the Ricci flow together with the desired local estimates on a correspondingly small spatial scale and time interval. Now assuming that $g$ satisfies  the hypotheses globally (that is, on all scales), scaling invariance allows one to apply the localized result to the rescaled metrics $R^{-2}g$ for large $R$, yielding a Ricci flow solution defined up to a uniform time $T>0$. Rescaling back produces a sequence of Ricci flow solutions with initial metric $g$ defined up to time $R^2T$, and the Arzel\`a--Ascoli theorem together with a diagonal argument yields a long-time Ricci flow solution emerging from the initial metric $g$ satisfying the global estimates stated in \eqref{generalconclusions}.   The volume ratio bound on all scales in this approach is required by the use of covering arguments  wherein one must apply a version of the ``Expanding Balls Lemma" (see {\cite[Lemma 3.7]{Fei}}, {\cite[Lemma 2.2]{LeeTam}}, {\cite[Lemma 4.2]{Martens}}, and {\cite[Lemma 2.3]{MartensScalar}} for the development of this lemma).   
  
    In contrast, the perspective taken in  \cite{ChenEric, ChauMartens} can be described as \begin{it} Global to Global \end{it} where typically a Ricci flow solution emerging from $g$ is known to exist and one works directly with global quantities over $M$ at \textit{each} time $t$.    A key challenge here is guaranteeing that the global curvature concentration is upper continuous in time and does not instantly `jump'.  Once this is established, it is relatively standard to show that the global curvature concentration must in fact be decreasing quantity along the flow, and a $C/t$ bound as in \eqref{generalconclusions} can be subsequently obtained via a Moser Iteration argument.

  The main purpose of this paper is to generalize the result in \cite{ChauMartens} by weakening the assumption that $g$ is complete with bounded curvature to assuming only that $g$ is bi-Lipschitz equivalent to a complete bounded curvature metric $h$ while also satisfying a Morrey-type integral condition on the gradient of $g$ relative $h$ as first considered in \cite{LeeLiu} in the context of Ricci flow.

\begin{thm}\label{thmmain}
For $n \ge 3$, there exists a dimensional constant $\delta(n) \geq \frac{1}{10^6 n^2}$ such that the following holds.  Let $(M^n,g)$ be a complete (not necessarily bounded curvature) Riemannian manifold satisfying all the following estimates:

\begin{enumerate}[(i)]
\item \label{SobMain} (\textbf{Sobolev Inequality})    
\[
\left( \int_M |u|^{\frac{2n}{n-2}}\, \dV_g \right)^{\frac{n-2}{n}}
\;\le\;
C_g \int_M |\nabla u|^2\,\dV_g ,
\]
for all compactly supported smooth functions $u$;

\item  \label{CCmain}(\textbf{Relatively Small Curvature Concentration})   
\[
\left( \int_M |\Rm|_g^{\frac{n}{2}}\,\dV_g \right)^{\frac{2}{n}}
\;\le\;
\delta(n)\, C_g^{-1};
\]
and

\item \label{morreyiii}(\textbf{Equivalence and Morrey–type condition})  
There exists a complete bounded curvature metric $h$ on $M$  and 
constants $\Lambda_0>1$, $L_0,\eta,r_0>0$, and $p\ge 1$ such that
\[
\left\{
\begin{aligned}
\Lambda_0^{-1}\, & h \;\le\; g \;\le\; \Lambda_0\, h 
\quad &&\text{ on } M, \\[6pt]
\dashint_{B_h(x_0,r)} |\widetilde \nabla g|_h^p& \,\dV_h
\;\le\;
L_0\, r^{-p+\eta}
\quad &&\text{for all } x_0\in M,\; 0<r<r_0,
\end{aligned}
\right.
\]
where $\widetilde{\nabla}$ denotes the covariant derivative with respect to $h$.

\end{enumerate}

Then there exists a complete long-time Ricci flow $g(t)$, $t\in [0,\infty)$, with initial condition $g(0)=g$ satisfying

\begin{enumerate}[(a)]
    \item \label{thma}Scaling invariant curvature estimates with explicit estimates: 
    \[
    \|\Rm\|_{g(t)}\leq \frac{(5e)^{n/2}}{t}
    \]
    \item Global Sobolev inequality for all times:
    \[
        \left( \int_M |u|^{\frac{2n}{n-2}}\, \dV_{g(t)} \right)^{\frac{n-2}{n}}
\;\le\;
(1000 n) C_g \int_M |\nabla u|_{g(t)}^2\,\dV_{g(t)}
\]
for all compactly supported smooth functions $u$. 
\item Decreasing curvature concentration:
\[
\left( \int_M | \Rm|_{g(t)}^{\frac{n}{2}}\,\dV_{g(t)} \right)^{\frac{2}{n}}\leq \left( \int_M  | \Rm|_{g(s)}^{\frac{n}{2}}\,\dV_{g(s)} \right)^{\frac{2}{n}}
\]
whenever $t\geq s\geq 0$. 
\item \label{thmd}Faster than $c/t$ curvature decay at infinity:
\[
\lim_{t\to\infty} 
\left( t\, \|\Rm\|_{g(t)}\right)
=0.
\]
\end{enumerate}
In particular, $M^n$ is diffeomorphic to $\mathbb{R}^n$.
\end{thm}
\begin{remark}
By classical results of Shi \cite{Shi} for Ricci flow of complete bounded curvature metrics, we may assume  without loss of generality in Theorem \ref{thmmain} that all covariant derivatives of the Riemann curvature tensor of $h$ are bounded on $M$.  In other words, $\|\nabla^k_h Rm\|_h<\infty$ for each $k\geq 0$.  
\end{remark}

\begin{remark}
    We have chosen to state the conclusions of Theorem \ref{thmmain} to be more similar to the main result of \cite{ChauMartens} rather than the general theorem given above. In particular, the conclusions as in \eqref{generalconclusions} (with the general $\epsilon$ and the injectivity estimate) could be mimicked here by modifying the constant $C_0(n)$.\end{remark}

 Under condition \ref{morreyiii} alone, it was proved in \cite{LeeLiu} that there exists a short-time complete solution $g(t)$ to the Ricci De-Turck flow relative to $h$ satisfying a curvature bound of the form 
    \[
     \|\Rm\|_{g(t)}\leq \frac{C}{t^{1-a}}
    \]
for and constants $C>0$ and $a>0$ (depending on $n,\Lambda_0,L_0,\eta, r_0,p,h$).  In particular, the bound is integrable in $t$ to $t=0$.  This will play a crucial role in our arguments which are in line with the Global to Global type described above.

This paper is organized as follows.  In Section \ref{SectionLp} we show that the $L^p$ curvature is upper continuous along any complete Ricci flow provided it is initially bounded and that the curvature is suitably integrable in time at $t=0$.  In Section \ref{SectionProof} we apply this result to prove Theorem \ref{thmmain} by adapting the proof of the main result of \cite{ChauMartens} in this setting and using the result in \cite{LeeLiu} . \\

\section{Upper-Continuity of $L^p$ curvature}\label{SectionLp}

In this section, we prove the following theorem saying that under a suitable integrability condition in time at $t=0$, the total $L^p$-curvature (at worst) grows continuously, provided it is initially bounded.

    \begin{thm}\label{upperContLp}
        For any $n\geq 3$, $C_0>0$ and $a\in (0,1)$, there exists $\widetilde T(n,a,C_0)>0$ such that the following holds. Let $(M,g(t))$, $t\in [0,T]$ be any complete Ricci flow (not assumed bounded curvature), which satisfies the integrable to time-0 curvature bound
        \begin{equation}\label{integrableCurveBoundLemma}
        \|\Rm\|_{g(t)}\leq \frac{C_0}{t^{1-a}}
        \end{equation}
        for all $t\in (0,T]$. Then the global $L^p$-curvature ($p\geq 1$) of the solution $g(t)$ satisfies
        \[
        \int_M |\mathrm{Rm}|_{g(t)}^p\,\dV_{g(t)}\leq e^{\frac{C_0(16p+n)}{a}t^a}\int_M |\mathrm{Rm}|_{g(0)}^p\,\dV_{g(0)}
        \]
        for all $t\in [0,\min\{T,\widetilde T\}]$, provided the RHS above is finite.
    \end{thm}

\begin{proof}
For clarity of proof, we begin by defining 
\[
\widetilde T=\min\{(nC_0)^{-1/a},\widehat T(n,1,3)\},
\]
where $\widehat T$ is from {\cite[Theorem 1.1]{LeeTam}}. Also for brevity, write $g=g(0)$. Fix some $p\geq 1$ and without loss of generality, assume that $|\Rm|_{g(0)}^p$ is integrable on $(M^n, g)$. Let $v_0 : M \to \mathbb{R}$ be an integrable function on $(M^n, g)$ satisfying $0 < v_0 \le 1$. For any $\varepsilon \in (0,1]$, define $v_{\varepsilon}(x, t)$ to be the solution to the initial value problem
\[
\begin{cases}
\partial_t v_{\varepsilon} = \Delta_{g(t)} v_{\varepsilon},\\
v_{\varepsilon}(x, 0) = \varepsilon v_0(x).
\end{cases}
\]
Also define
\[
K_{\varepsilon}^2 := | \mathrm{Rm} |_{g(t)}^2 + v_{\varepsilon}^2.
\]

In {\cite[Lemma 2.1]{ChauMartens}} the following estimate was shown using the standard estimate for the evolution of the curvature tensor along Ricci flow:
\[
\partial_t K_{\varepsilon}^p \le \Delta_{g(t)} K_{\varepsilon}^p + 8 p K_{\varepsilon}^{p+1},
\]
which, after applying the curvature bound in \eqref{integrableCurveBoundLemma}, we can simplify to
\begin{equation}\label{evolOfKp}
\partial_t K_{\varepsilon}^p \le \Delta_{g(t)} K_{\varepsilon}^p + 16C_0p t^{-1+a}\, K_{\varepsilon}^{p}.
\end{equation}
Note, the addition of $v_{\varepsilon}^2$ above is only needed to prevent possible division by zero in deriving the above estimates in the case when $p<2$.  Thus if we define $\widetilde K_{\varepsilon}^p:=e^{-(16C_0p/a)t^{a}}K_{\varepsilon}^p$, then $\widetilde K_\epsilon^p$ is a subsolution to the time dependent heat equation:  

\begin{equation}\label{upperlowerK}
\partial \widetilde K_{\varepsilon}^p  \le  \Delta_{g(t)} \widetilde K_{i,\varepsilon}^p \text{   on } M\times[0, T].
\end{equation}

Now we fix a sequence of nested compact sets $\{\Omega_j\}_{j=0}^\infty$ exhausting $M$, and for each $j$, let $\phi_j: M \to [0,1]$ be a smooth cutoff function  supported in $\Omega_j$ and with $\phi_j \equiv 1$ in $\Omega_{j-1}$. For any $\varepsilon > 0$ and $j$ we denote by 
\[
u_{j} : \Omega_j \times [0,T] \to \mathbb{R}
\] 
the solution to the initial value Dirichlet boundary value problem
\[
\begin{cases}
\partial_t u_{j} = \Delta_{g(t)} u_{j}, & \text{on }  \Omega_t\times[0, T] \\
u_{j}(x,0) = \phi_j K_{\varepsilon}^p(x,0), & x \in \Omega_j\\
u_{j}(x,t) = 0, & x \in \partial \Omega_j.
\end{cases}
\]
Here the dependence of $u_j$ on $\varepsilon$ has been suppressed for simplicity and will not cause confusion for our purpose.  Just as in the proof of Proposition 2.2 in \cite{ChauMartens}, we obtain the evolution inequality
\[
\frac{\mathrm{d}}{\mathrm{d}t}\int_{\Omega_j} u_{j} \,\dV_{g(t)} \le n\int_{\Omega_j} u_{j}|\mathrm{Rm}|_{g(t)} \,\dV_{g(t)}\le nC_0 t^{-1+a} \int_{\Omega_j} u_{j} \,\dV_{g(t)},
\]
where we have used \eqref{integrableCurveBoundLemma} and the scalar curvature bound $|\mathrm{Scal}|\leq n|\Rm|$. Thus by comparison to the corresponding ODE, we have
\[
\int_{\Omega_j} u_{j}(x, t) \,\dV_{g(t)}\leq e^{(nC_0/a)t^{a}}\int_{\Omega_j} u_{j}(x, 0) \,\dV_{g(0)} \le e^{(nC_0/a)t^{a}}\int_{M} |\mathrm{Rm}|_{g(0)}^p  \,\dV_{g(0)} .
\]

By the classical maximum principle, the sequence $u_j (x, t)$ is non-decreasing in $j$ for any fixed $(x, t)$.  On the other hand, the sequence $u_j(x, t)$ must be bounded above on any compact set of $M$ independent of $j, t$ as otherwise the integral upper bound above would be violated by the classical Harnack inequality for linear parabolic equations in \cite{KrylovSafanov}.  It follows that $u_j(x, t)$ converges as $j\to \infty$ in $L^{\infty}_{loc}(M\times[0, T])$ to a limit $u(x, t)$, and by the parabolic Schauder estimates it follows that $u(x, t)\in C^{\infty}$ and solves

\[
\begin{cases}
\partial_t u= \Delta_{g(t)} u\; \text{ on }M\times[0, T] & \\
u_{j}(x,0) = K_{\varepsilon}^p(x,0), \text{ on }M,& \\
\end{cases}
\]
while satisfying
\begin{equation}\label{integrableu}
\int_{M} u(x, t) \,\dV_{g(t)} \le e^{(nC_0/a)t^{a}}\int_{M} K_{\varepsilon}^p(x,0)  \,\dV_{g(0)} .
\end{equation}
From the above equation for $u$ and \eqref{upperlowerK} we get that the function $\varphi^R(x,s):=\widetilde K_{\varepsilon}^p(x,Rs)-u(x,Rs)$ satisfies
\begin{equation}\label{heatphiR}
\begin{cases}
\partial_s \varphi^R\leq \Delta_{g_R(t)} \varphi^R \; \text{ on }M\times[0, T/R] & \\
\varphi^R(x, 0)=0, \text{ on }M& \\
\end{cases}
\end{equation}
for any $R>0$ where $g_R(x, t):=g(x, Rt)/R$ is also a solution to the Ricci $\eqref{RF}$ (with initial condition $g(0)/R$).
\,

For simplicity of notation, we will – for the remainder of the proof – assume that $T\leq \widetilde T$. Note that the integral bound \eqref{integrableCurveBoundLemma}, together with the bound $\widehat T\leq (nC_0)^{-1/a}$, implies the scale-invariant curvature bound
\[
\|\Rm\|_{g(t)}\leq \frac{1}{nt} \quad \implies \quad \Rc_{g(t)}\leq \frac{g(t)}{t}
\]
on $(0,T]$.  Now recall that we are assuming $T\leq \widetilde{T}\leq \widehat T$ where $\widehat T = \widehat T(n, 1, 3)$ is from the conclusion of the local maximum principle in {\cite[Theorem 1.1]{LeeTam}} with $\alpha, l$ there being $1,3$ respectively.  From that theorem, it follows that whenever $R\geq 1$, the statement of that theorem gives at any point $p\in M$ the inequality
\[
\varphi^R(p,s)\leq s^3 \,\,\, \text{for all} \,\,\, s \in [0, R^{-1}T].
\]
Fixing any $t_0\in [0, T]$ and evaluating the above at $s_R:=t_0/R$ gives 
\[
\widetilde K_{\varepsilon}^p(p,t_0)-u(p,t_0)\leq t_0^3 R^{-3}. 
\]
Letting $R\to \infty$ and noting $(p, t_0)$ was chosen arbitrarily, we conclude that
\begin{equation}\label{upperboundu}\widetilde K_{\varepsilon}^p-u\leq 0 \,\,\, \text{on} \,\,\, M\times[0, T].
\end{equation}
\,

Thus, using \eqref{integrableu} we conclude 

\[\int_{M} \widetilde K_{\varepsilon}^p(x, t) \,\dV_{g(t)} \le e^{(nC_0/a)t^{a}}\int_{M} K_{\varepsilon}^p(x,0)  \,\dV_{g(0)} \]
for all $t\in [0, T]$. The proof is then readily completed by recalling $\widetilde K_{\varepsilon}^p = e^{-(16C_0 p/a)t^a)}K_{\varepsilon}^p$ and  taking $\epsilon\to 0$. \\
\end{proof}

Before concluding this section, we take the time to outline a slightly altered argument to the above in the event that the (possibly unbounded curvature) flow $g(t)$ is obtained as a local smooth limit of \textbf{bounded} curvature flows $g_i(t)$ (this is indeed the setting of our main theorem as desribed in the next section). In this setting, we are able to trade the Harnack inequality argument in the above for some careful treatment of limits. \\

We assume in this outline that the $g_i(t)$ are a sequence of bounded curvature Ricci flows, all on a uniform time interval $t\in [0,T]$, each satisfying the integrable curvature decay estimate 
\[
\|\Rm\|_{g_i(t)}\leq \frac{C_0}{t^{1-a}}.
\]
Moreover, assume $g_i(0)=g_0$ on $\Omega_i$ (defined above), and that 
\[
g_i(t) \xrightarrow{C^\infty_{loc}(M\times [0,T])} g(t).
\]
Note, in the subsequent proof, it is not necessary for each $g_i(t)$ to solve the Ricci flow \eqref{RF} on the whole space $M\times[0, T)$, but we will assume so for simplicity. Let $\phi_j$ be as in the above proof and relative to this and $g_i(t)$, define the functions $K_{i,\varepsilon}$, $u_{ij}$, $\varphi_{ij}$ $\widetilde K_{i,\varepsilon}$ also as in the above proof. \\

Then, recalling that $\phi_j=1$ on $\Omega_{j-1}$, we can say that as long as $i\geq j+1$, then we know in particular that $\widetilde K_{i,\varepsilon}^p(x,0)=u_{ij}(x,0)$, and therefore that
\begin{equation}\label{heatphiIJ} \partial_t  \varphi_{ij}\leq \Delta_{g_i(t)} \varphi_{ij}\quad \text{ and } \quad  \varphi_{ij}(x, 0)=0\;\text{ on }\Omega_j \quad  \text{ whenever } i\geq j+1.
\end{equation} 
Now fix  $p\in M$ and large $R\geq 1$. Whenever $J_R(p)$ is large enough so that $d_{g_0}(p,\partial \Omega_j)\geq 2R$ for all $j\geq J_R(p)$, we can similarly apply the local maximum principle {\cite[Theorem 1.1]{LeeTam}} to obtain
\begin{equation}
\varphi_{ij}(p,t)\leq t^3R^{-6} \; \text{ for all }\; t\in [0,T], j\geq J_R(p), i\geq j+1.
\end{equation}
Taking the limit inferior in $i$, followed by $j$, then the limit $R\to \infty$, we obtain the following inequality of functions:
\[
\liminf_{j\to \infty}\liminf_{i\to \infty}\widetilde K_{i,\varepsilon}^p\leq \liminf_{j\to \infty}\liminf_{i\to \infty} u_{ij}.
\]

Recalling that $\widetilde K_{i,\varepsilon}^p=e^{-(16C_0p/a)t^{a}}K_{i,\varepsilon}^p$ and integrating the above inequality over $M$ with respect to $g(t)$ gives
\[
    \begin{split}
&\quad e^{-(16C_0p/a)t^{a}}\int_M \liminf_{j\to \infty}\liminf_{i\to \infty} K_{i,\varepsilon}^p \;\dV_{g(t)}\leq
\int_M \liminf_{j\to \infty}\liminf_{i\to \infty} u_{ij}\,\dV_{g_i(t)}\leq
\liminf_{j\to \infty}\liminf_{i\to \infty}\int_{\Omega_j}  u_{ij}\,\dV_{g_i(t)}
\end{split}
\]
by Fatou's Lemma. The result can be concluded using the uniform bound
\[
    \int_{\Omega_j} u_{ij} \,\dV_{g_i(t)}\leq e^{(nC_0/a)t^{a}}\int_M \bigl(|\mathrm{Rm}|^2_{g_0} + \varepsilon^2 v_0^2\bigr)^{p/2} \, \,\dV_{g_0} \quad \text{ whenever } i\geq j.
\]

\section{Proof of Theorem \ref{thmmain}}\label{SectionProof}

\begin{proof}[Proof of Theorem \ref{thmmain}]

 In  \cite{LeeLiu} it was proved that given a smooth metric $g_0$ satisfying condition \ref{morreyiii} in Theorem \ref{thmmain}, there exists a solution $\hat{g}(t)$ to the Ricci De-Turck flow relative to $h$ on $M\times[0, T]$ for some $T>0$:
 \begin{equation}\label{LLmainthm}
\frac{\partial}{\partial t} \hat{g}(t)=-2\,\mathrm{Rc}(\hat{g}(t))-
\mathcal{L}_{X_h(\hat{g}(t))} \hat{g}(t),
\end{equation}
 where at any time $t$, $\mathcal{L}_{X_h(\hat{g}(t))} g(t)$ denotes the Lie derivative of the metric $\hat{g}(t)$ relative to the vector field $X_h$, whose local components (suppressing the dependence on $t$) are given by 
  \[
[X_h(g)]^k := \hat{g}^{ij}\left(\tilde{\Gamma}^k_{ij} - \Gamma^k_{ij}\right).
\]
More precisely, $\hat{g}(t)$ was constructed as a $C^{\infty}_{loc}(M\times[0, T])$ limit of a sequence
$\{\hat{g}_i(t)\}$ of complete bounded curvature solutions to the above Ricci De-Turck flow on $M\times[0, T]$ (with $\hat{g}$ replaced with $\hat{g}_i$) and with initial condition 
\[
g_{i,0} = \phi(\rho/R_i) g_0 + \bigl(1-\phi(\rho/R_i)\bigr)h.
\]
Here $\rho:M\to \mathbb{R}$ is some smooth distance like function on $M$ satisfying
\[
|\nabla_h\rho|_h^2 + |\nabla_h^2 \rho|_h \le 1
\]
and
\[
C(n,k_0)^{-1}\bigl(d_h(\cdot,p)+1\bigr)
\le \rho(\cdot) \le
C(n,k_0)\bigl(d_h(\cdot,p)+1\bigr)
\]
for some $C(n,k_0)>0$, where $k_0$ is such that $\|\mathrm{Rm}\|_h \le k_0$, and $\phi$ is some cut off function on $[0,\infty)$ such that $\phi \equiv 1$ on $[0,1]$, $\phi \equiv 0$ on $[2,+\infty)$, and $0 \le -\phi' \le 10$ and $R_i\to \infty$ is some increasing sequence.  In particular, it was shown that each $\hat{g}_i(t)$ is a bounded curvature solution to the above Ricci De-Turck flow on $M\times[0, T]$ for some $T$ independent of $i$ and satisfying
\[
 \|\Rm\|_{g_i(t)}\leq \frac{C_0}{t^{1-a}}.
\]
on $M\times[0, T]$ for some $C_0$ and $a>0$ independent of $i$ (see {\cite[Lemma 2.2]{LeeLiu}} and {\cite[Proposition 2.4]{LeeLiu}}).

 Now if for each $i$ we let $\Psi_{i, t} :M\to M$ be the time dependent family of diffeomorphisms with $\Psi_{i, 0}=\mathrm{Id}$ and local components functions satisfying $$
 \frac{\partial}{\partial t} [\Psi_{i,t}]^k = [X_h(\hat{g}_i(t))]^k$$ 
for $t\geq 0$ then it follows that $$g_i(t):=\Psi_{i,t} ^* \hat{g}_i(t)$$
solves Ricci flow \eqref{RF} on $M\times[0, T]$ while also satisfying the above curvature estimates. Moreover, it follows by the results of Chen \cite{Chen} that a subsequence of $g_i(t)$ converges in $C^{\infty}_{loc}(M\times[0, T])$ to a smooth solution $g(t)$ to Ricci flow \eqref{RF} on $M\times[0, T]$ also satisfying 

\begin{equation}\label{RFboundX}
\|\mathrm{Rm}\|_{g (t)}\leq \frac{C_0}{t^{1-a}}.
\end{equation}

This integrability of curvature as $t\to 0$ implies that the metrics $g(t)$ converge to $g(0)$ in $C^0(M)$ as $t\to 0$. As a consequence, the metrics $g(t)$ will satisfy \eqref{Sob} with some constants $C_t$ converging to $C_g$ as $t\to 0$. Moreover, the established curvature bound \eqref{RFboundX} allows us to apply Theorem \ref{upperContLp} to see that the (global) $L^{n/2}$-curvature of the $g(t)$ are upper-continuous. Thus, similarly to above, the $g(t)$ will satisfy \eqref{CC} with some constants $\sigma_t\to \sigma_0$ as $t\to 0$.  Then with $\delta_n$ being the dimensional constant in {\cite[Theorem 1.2]{ChauMartens}}, notice that
\[
\sigma_0=\delta(n)C_g^{-1}=(10^6 n^2 C_g)^{-1}< \delta_n C_g^{-1}.
\] 
In particular, this result may thus be applied to the complete bounded curvature metric $g(t)$, for any $t$ sufficiently small. Each of the desired conclusions \ref{thma} - \ref{thmd} follow from the proof of {\cite[Theorem 1.2]{ChauMartens}}, while the diffeomorphism to $\R^n$ follows (as mentioned in the introduction) by applying {\cite[Theorem 1.1]{HuangPeng}} to some shifted-time metric $g(T)$ along the flow.

\end{proof}

\end{document}